\documentclass[leqno,12pt]{article} 
\usepackage{amsmath, amssymb}
\usepackage{amsthm} 
\theoremstyle{plain} 
\newtheorem{theorem}{\indent\sc Theorem}[section]
\newtheorem{lemma}[theorem]{\indent\sc Lemma}
\newtheorem{corollary}[theorem]{\indent\sc Corollary}
\newtheorem{proposition}[theorem]{\indent\sc Proposition}

\theoremstyle{definition} 

\newtheorem{remark}[theorem]{\indent\sc Remark}

\title{Harmonic aspects in an $\eta$-Ricci soliton}
\author{Adara M. Blaga}

\date{}
\begin{document}

\maketitle

\markboth{{\small\it {\hspace{4cm} Harmonic aspects in an $\eta$-Ricci soliton}}}{\small\it{Harmonic aspects in an $\eta$-Ricci soliton
\hspace{4cm}}}

\footnote{2010 Mathematics Subject Classification: 35C08, 53C25}
\footnote{Keywords and phrases: gradient Ricci solitons, Schr\"{o}dinger-Ricci equation, harmonic form}


\begin{abstract}
We characterize $\eta$-Ricci solitons $(g,\xi,\lambda,\mu)$ in some special cases when the $1$-form $\eta$, which is the $g$-dual of $\xi$, is a harmonic or a Schr\"{o}dinger-Ricci harmonic form. We also provide necessary and sufficient conditions for $\eta$ to be a solution of the Schr\"{o}dinger-Ricci equation and point out the relation between the three notions in our context. In particular, we apply these results to a perfect fluid spacetime and using Bochner-Weitzenb\"{o}ck techniques, we formulate some more conclusions for gradient solitons and deduce topological properties of the manifold and its universal covering.
\end{abstract}

\maketitle


\section{Introduction}

Self-similar solutions to the Ricci flow, the \textit{Ricci solitons} \cite{ham} have been studied in different geometrical contexts on complex, contact and paracontact manifolds. The more general
notion of \textit{$\eta$-Ricci soliton} was introduced by J. T. Cho and M. Kimura \cite{ch} on real hypersurfaces in a K\"{a}hler manifold and treated in complex space forms \cite{cacr}, Euclidean hypersurfaces \cite{b12}, paracontact geometries \cite{bla}, \cite{bl}, \cite{blcr}, \cite{b8}, \cite{b9}, \cite{b10}. Different geometrical aspects of Ricci and $\eta$-Ricci solitons have been studied by author in \cite{b11}, \cite{b1}, \cite{b6}. Further generalizations of this notion and properties of other geometrical solitons can be found in \cite{b3}, \cite{b7} and \cite{b4}, \cite{b5}.

A particular case of solitons arise when they evolve by diffeomorphism generated by a gradient vector field, namely when
the potential vector field is the gradient of a smooth function. The gradient vector fields play a central r\^{o}le in Morse-Smale theory \cite{ms} and some aspects of gradient $\eta$-Ricci solitons were discusses by author in \cite{b}, \cite{blaga}, \cite{ada}, \cite{blag}, \cite{blagam}, \cite{b2}.

In Section 2, after we point out the basic properties of an $\eta$-Ricci soliton $(g,\xi,\lambda,\mu)$, we provide necessary and sufficient conditions for the $g$-dual $1$-form of the potential vector field $\xi$ to be a solution of the Schr\"{o}dinger-Ricci equation, a harmonic or a Schr\"{o}dinger-Ricci harmonic form and characterize the $1$-forms orthogonal to $\eta$. We end these considerations by discussing the case of a perfect fluid spacetime. In Section 3 we formulate the results for the special case of gradient solitons and deduce topological properties of the manifold and its universal covering \cite{l}.

\section{Geometrical aspects of $\eta$}

Let $(M,g)$ be an $n$-dimensional Riemannian manifold, $n> 2$, and denote by $\flat:TM\rightarrow T^*M$, $\flat(X):=i_Xg$, $\sharp:T^*M\rightarrow TM$, $\sharp:=\flat^{-1}$ the musical isomorphism. Consider the set $\mathcal{T}^0_{2,s}(M)$ of symmetric $(0,2)$-tensor fields on $M$ and for $Z\in \mathcal{T}^0_{2,s}(M)$, denote by $Z^{\sharp}:TM\rightarrow TM$ and by $Z_{\sharp}:T^*M\rightarrow T^*M$ the maps defined as follows:
$$g(Z^{\sharp}(X),Y):=Z(X,Y), \ \ Z_{\sharp}(\alpha)(X):=Z(\sharp(\alpha),X).$$
We also denote by $Z^{\sharp}$ the map $Z^{\sharp}:T^*M\times T^*M\rightarrow C^{\infty}(M)$:
$$Z^{\sharp}(\alpha,\beta):=Z(\sharp(\alpha),\sharp(\beta))$$
and can identify $Z_{\sharp}$ with the map also denoted by $Z_{\sharp}:T^*M\times TM\rightarrow C^{\infty}(M)$:
$$Z_{\sharp}(\alpha,X):=Z_{\sharp}(\alpha)(X).$$

Given a vector field $X$, its $g$-dual $1$-form $X^{\flat}=:\flat(X)$ is said to be a \textit{solution of the Schr\"{o}dinger-Ricci equation} if it satisfies:
\begin{equation}
div(L_Xg)=0,
\end{equation}
where $L_Xg$ denotes the Lie derivative along the vector field $X$.

It is known that \cite{cho}:
\begin{equation}\label{e61}
div(L_Xg)=(\Delta+S_{\sharp})(X^{\flat})+d(div(X)),
\end{equation}
where $\Delta$ denotes the Laplace-Hodge operator on forms w.r.t. the metric $g$ and $S$ the Ricci curvature tensor field. Denoting by $Q$ the Ricci operator defined by $g(QX,Y):=S(X,Y)$, for any vector fields $X$ and $Y$, by a direct computation we deduce that $S_{\sharp}(\gamma)=i_{Q\gamma^{\sharp}}g$, for any $1$-form $\gamma$.

\subsubsection*{\textbf{$\eta$-Ricci solitons.}}

We are interested to find the necessary and sufficient conditions for the $g$-dual $1$-form $\eta$ of the potential vector field $\xi$ in an $\eta$-Ricci soliton to be a solution of the Schr\"{o}dinger-Ricci equation, a harmonic or Schr\"{o}dinger-Ricci harmonic form.

\bigskip

Consider the equation:
\begin{equation}\label{11}
L_{\xi}g+2S+2\lambda g+2\mu\eta\otimes \eta=0,
\end{equation}
where $g$ is a Riemannian metric, $S$ its Ricci curvature tensor field, $\xi$ a vector field, $\eta$ a $1$-form and $\lambda$ and $\mu$ are real constants. The data $(g,\xi,\lambda,\mu)$ which satisfy
the equation (\ref{11}) is said to be an \textit{$\eta$-Ricci soliton} on $M$ \cite{ch}; in particular, if $\mu=0$, $(g,\xi,\lambda)$ is a \textit{Ricci soliton} \cite{ham} and it is called \textit{shrinking}, \textit{steady} or \textit{expanding} according as $\lambda$ is negative, zero or positive, respectively \cite{chlu}. If the potential vector field $\xi$ is of gradient-type, $\xi=grad(f)$, for $f$ a smooth function on $M$, then $(g,\xi,\lambda,\mu)$ is called a \textit{gradient $\eta$-Ricci soliton}.

\bigskip

Taking the trace of the equation (\ref{11}) we obtain:
\begin{equation}\label{eq2}
div(\xi)+scal+\lambda n+\mu |\xi|^2=0.
\end{equation}

From a direct computation we get:
$$div(\eta\otimes\eta)=div(\xi)\eta+\nabla_{\xi}\eta.$$

Now taking the divergence of (\ref{11}) and using (\ref{e61}) we obtain:
\begin{equation}\label{eq1}
div(L_{\xi}g)+d(scal)+2\mu [div(\xi)\eta+\nabla_{\xi}\eta]=0.
\end{equation}

\subsubsection*{\textbf{Schr\"{o}dinger-Ricci solutions.}}

We say that a $1$-form $\gamma$ is a \textit{solution of the Schr\"{o}dinger-Ricci equation} if
\begin{equation}\label{e111}
(\Delta+S_{\sharp})(\gamma)+d(div(\gamma^{\sharp}))=0.
\end{equation}

\begin{theorem}\label{t1}
Let $(g,\xi,\lambda,\mu)$ be an $\eta$-Ricci soliton on the $n$-dimensional manifold $M$ with $\eta$ the $g$-dual of $\xi$. Then $\eta$ is a solution of the Schr\"{o}dinger-Ricci equation if and only if
\begin{equation}\label{e211}
d(scal)=2\mu[(scal+\lambda n+\mu|\xi|^2)\eta-\nabla_{\xi}\eta].
\end{equation}

Moreover, in this case, $scal$ is constant if and only if $\mu=0$ (which yields a Ricci soliton) or $(scal+\lambda n+\mu|\xi|^2)\eta=\nabla_{\xi}\eta$.
\end{theorem}
\begin{proof}
From (\ref{11}), (\ref{eq2}), (\ref{eq1}) and
$$2div(S)=d(scal)$$
it follows that $\eta$ is a solution of the Schr\"{o}dinger-Ricci equation if and only if (\ref{e211}) holds.
\end{proof}

\begin{remark}
Under the hypotheses of Theorem \ref{t1}, if the potential vector field is of constant length $k$, then from (\ref{e211}) we deduce that the scalar curvature is constant if either the soliton is a Ricci soliton or, $(scal+\lambda n+\mu k^2)\eta=\nabla_{\xi}\eta$ which implies $scal=-\lambda n-\mu k^2$.
\end{remark}

\begin{corollary}\label{c}
Let $(g,\xi,\lambda,\mu)$ be an $\eta$-Ricci soliton on the $n$-di\-men\-sio\-nal manifold $M$ with $\eta$ the $g$-dual of $\xi$ and assume that $\eta$ is a nontrivial solution of the Schr\"{o}dinger-Ricci equation. If $scal$ is constant and $\mu\neq 0$, then $\frac{1}{2|\xi|^2}\xi(|\xi|^2)-\mu |\xi|^2=scal+\lambda n$ (constant).
\end{corollary}
\begin{proof}
Under the hypotheses, from (\ref{e211}) we obtain:
$$(scal+\lambda n+\mu|\xi|^2)\eta=\nabla_{\xi}\eta,$$
applying $\xi$ and taking into account that
$$(\nabla_{\xi}\eta)\xi=\frac{1}{2}\xi(|\xi|^2),$$
we deduce that $(scal+\lambda n+\mu|\xi|^2)|\xi|^2=\frac{1}{2}\xi(|\xi|^2)$.
\end{proof}

For the case of Ricci solitons, from Theorem \ref{t1} we have:
\begin{corollary}
If $(g,\xi,\lambda)$ is a Ricci soliton on the $n$-di\-men\-sio\-nal manifold $M$ and $\eta$ is the $g$-dual of $\xi$, then $\eta$ is a solution of the Schr\"{o}dinger-Ricci equation if and only if the scalar curvature of the manifold is constant.
\end{corollary}

\subsubsection*{\textbf{Schr\"{o}dinger-Ricci harmonic forms.}}

We say that a $1$-form $\gamma$ is \textit{Schr\"{o}dinger-Ricci harmonic} if
$$(\Delta+S_{\sharp})(\gamma)=0.$$

From (\ref{e111}), (\ref{eq2}) and (\ref{eq1}) we deduce:
\begin{theorem}\label{te}
Let $(g,\xi,\lambda,\mu)$ be an $\eta$-Ricci soliton on the $n$-dimensional manifold $M$ with $\eta$ the $g$-dual of $\xi$. Then $\eta$ is a Schr\"{o}dinger-Ricci harmonic form if and only if $\mu=0$ (which yields a Ricci soliton) or
\begin{equation}\label{e4}
(scal+\lambda n+\mu|\xi|^2)\eta=\nabla_{\xi}\eta-\frac{1}{2}d(|\xi|^2).
\end{equation}
\end{theorem}

\begin{remark}\label{r}
Under the hypotheses of Theorem \ref{te}, if $\mu\neq 0$, then from (\ref{e4}) we deduce that the scalar curvature is constant if and only if the potential vector field is of constant length.
\end{remark}

\subsubsection*{\textbf{Harmonic forms.}}

We know that on a Riemannian manifold $(M,g)$, a $1$-form $\gamma$ is \textit{harmonic} (i.e. $\Delta(\gamma)=0$) if and only if it is closed and divergence free.

Remark that on an $\eta$-Ricci soliton, a harmonic $1$-form $\gamma$ is Schr\"{o}dinger-Ricci harmonic if and only if
$$\gamma\circ \nabla \xi+\lambda \gamma+\mu \gamma(\xi)\eta=0$$
which implies (using the fact that $(\nabla_X\gamma)^{\sharp}=\nabla_X\gamma^{\sharp}$, for any vector field $X$ and any $1$-form $\gamma$):
$$\gamma^{\sharp}\in \ker [\nabla_{\xi}\eta+(\lambda+\mu|\xi|^2)\eta].$$

From (\ref{e61}) and (\ref{eq1}) we deduce:
\begin{theorem}\label{teo}
Let $(g,\xi,\lambda,\mu)$ be an $\eta$-Ricci soliton on the $n$-dimensional manifold $M$ with $\eta$ the $g$-dual of $\xi$. Then $\eta$ is a harmonic form if and only if
\begin{equation}\label{p}
i_{Q\xi}g=\mu\{2[(scal+\lambda n+\mu|\xi|^2)\eta-\nabla_{\xi}\eta]+d(|\xi|^2)\}.
\end{equation}
\end{theorem}

For the case of Ricci solitons, from Theorem \ref{teo} we have:
\begin{corollary}
If $(g,\xi,\lambda)$ is a Ricci soliton on the $n$-di\-men\-sio\-nal manifold $M$ and $\eta$ is the $g$-dual of $\xi$, then $\eta$ is a harmonic form if and only if $\xi\in \ker Q$.
\end{corollary}

From (\ref{eq2}), (\ref{e4}) and (\ref{p}) we deduce:
\begin{corollary}\label{t4}
Let $(g,\xi,\lambda,\mu)$ be an $\eta$-Ricci soliton on the $n$-dimensional manifold $M$ with $\eta$ the $g$-dual of $\xi$. If $\eta$ is a harmonic form, then i) $\xi\in \ker Q$ and ii) the scalar curvature is constant if and only if the po\-ten\-tial vector field $\xi$ is of constant length.
\end{corollary}

The relation between the cases when $\eta$ is a solution of the Schr\"{o}dinger-Ricci equation, harmonic or the Schr\"{o}dinger-Ricci harmonic form is stated in the following result:
\begin{lemma}
Let $(g,\xi,\lambda,\mu)$ be an $\eta$-Ricci soliton on the $n$-dimensional manifold $M$ with $\eta$ the $g$-dual of $\xi$.\\
i) If $\eta$ is a solution of the Schr\"{o}dinger-Ricci equation, then $\eta$ is:

a) Schr\"{o}dinger-Ricci harmonic form if and only if $scal+\mu|\xi|^2$ is constant;

b) harmonic form if and only if $i_{Q\xi}g=d(scal+\mu|\xi|^2)$; also $\eta$ harmonic implies $\xi\in\ker Q$.\\
ii) If $\eta$ is Schr\"{o}dinger-Ricci harmonic form, then $\eta$ is:

a) a solution of the Schr\"{o}dinger-Ricci equation if and only if $scal+\mu|\xi|^2$ is constant;

b) harmonic form if and only if $\xi\in\ker Q$.\\
iii) If $\eta$ is a harmonic form, then $\eta$ is:

a) a solution of the Schr\"{o}dinger-Ricci equation if and only if $\xi\in\ker Q$;

b) Schr\"{o}dinger-Ricci harmonic form if and only if $\xi\in\ker Q$.
\end{lemma}

We can synthesize:

i) if $scal+\mu|\xi|^2$ is constant, then $\eta$ is Schr\"{o}dinger-Ricci harmonic if and only if it is a solution of the Schr\"{o}dinger-Ricci equation;

ii) if $\xi\in\ker Q$, then $\eta$ is Schr\"{o}dinger-Ricci harmonic if and only if it is harmonic.

\subsubsection*{\textbf{$1$-forms orthogonal to $\eta$.}}

We say that two $1$-forms $\gamma_1$ and $\gamma_2$ are \textit{orthogonal} if $g(\gamma_1^{\sharp},\gamma_2^{\sharp})=0$ (i.e. $\langle\gamma_1,\gamma_2\rangle=0$, where $\langle\gamma_1,\gamma_2\rangle:=\sum_{i=1}^n\gamma_1(E_i)\gamma_2(E_i)$, for $\{E_i\}_{1\leq i\leq n}$ a local orthonormal frame field).

Remark that $\gamma_1$ and $\gamma_2$ are orthogonal if and only if $$\gamma_1^{\sharp}\in \ker \gamma_2\ \ \textit{or} \ \ \gamma_2^{\sharp}\in \ker \gamma_1.$$

\begin{theorem}\label{j}
Let $(g,\xi,\lambda,\mu)$ be an $\eta$-Ricci soliton on the $n$-dimensional manifold $M$ with $\eta$ the $g$-dual of $\xi$ and $\mu\neq 0$. If $\gamma$ is a $1$-form, then $\gamma$ is orthogonal to $\eta$ if and only if
\begin{equation}\label{gg}
\nabla_{\gamma^{\sharp}}\xi+Q\gamma^{\sharp}+\lambda \gamma^{\sharp}\in \ker \gamma.
\end{equation}
\end{theorem}
\begin{proof}
Observe that computing the soliton equation in $(\gamma^{\sharp},\gamma^{\sharp})$ and using the orthogonality condition we obtain:
\begin{equation}
g(\nabla_{\gamma^{\sharp}}\xi,\gamma^{\sharp})+g(Q\gamma^{\sharp},\gamma^{\sharp})+\lambda |\gamma^{\sharp}|^2=0
\end{equation}
which is equivalent to the condition (\ref{gg}).
\end{proof}

\subsubsection*{Example}

We end these considerations by discussing the case of a perfect fluid spacetime $(M,g,\xi)$ \cite{blagam}. If we denote by $p$ the isotropic pressure, $\sigma$ the energy-density, $\lambda$ the cosmological constant, $k$ the gravitational constant, $S$ the Ricci curvature tensor field and $scal$ the scalar curvature of $g$, then \cite{blagam}:
\begin{equation}\label{e5}
S=-(\lambda-\frac{scal}{2}-kp)g+k(\sigma+p)\eta\otimes \eta
\end{equation}
and the scalar curvature of $M$ is:
\begin{equation}
scal=4\lambda+k(\sigma-3p).
\end{equation}

From Theorem \ref{t1}, we deduce that if $(g,\xi,a,b)$ is an $\eta$-Ricci soliton on $(M,g,\xi)$, then $\eta$ is a solution of the Schr\"{o}dinger-Ricci equation if and only if
$$kd(\sigma-3p)=2b\{[4(a+\lambda)-b+k(\sigma-3p)]\eta-\nabla_{\xi}\eta\}.
$$
Moreover, the fluid is a radiation fluid (i.e. $\sigma=3p$) if and only if $b=0$ (which yields the Ricci soliton) or $[4(a+\lambda)-b]\eta=\nabla_{\xi}\eta$ which implies $b=4(a+\lambda)$.

From Theorem \ref{te}, we deduce that if $(g,\xi,a,b)$ is an $\eta$-Ricci soliton on $(M,g,\xi)$, then $\eta$ is a Schr\"{o}dinger-Ricci harmonic form if and only if $b=0$ (which yields a Ricci soliton) or
$$[4(a+\lambda)-b+k(\sigma-3p)]\eta=\nabla_{\xi}\eta$$
which implies $b=4(a+\lambda)+k(\sigma-3p)$.

From Theorem \ref{teo}, we deduce that if $(g,\xi,a,b)$ is an $\eta$-Ricci soliton on $(M,g,\xi)$, then $\eta$ is a harmonic form if and only if
$$\{4b[4(a+\lambda)-b+k(\sigma-3p)]-2\lambda+k(\sigma+3p)\}\eta=4b\nabla_{\xi}\eta.$$
For the case of Ricci soliton $(g,\xi,a)$ in a radiation fluid we obtain the constant pressure $p=\frac{\lambda}{3k}$.

\section{Applications to gradient solitons}

Let $f\in C^{\infty}(M)$, $\xi:=grad(f)$, $\eta:=\xi^{\flat}$ and $\lambda$ and $\mu$ real constants. Then $\eta=df$ and
\begin{equation}\label{e222}
g(\nabla_X\xi,Y)=g(\nabla_Y\xi,X),
\end{equation}
for any $X$, $Y\in\mathfrak{X}(M)$. Also \cite{bl}:
\begin{equation}
trace(\eta\otimes \eta)=|\xi|^2,
\end{equation}
\begin{equation}\label{rr}
div(\eta\otimes\eta)=div(\xi)\eta+\frac{1}{2}d(|\xi|^2)
\end{equation}
and
\begin{equation}\label{hh}
\nabla_{\xi}\eta=\frac{1}{2}d(|\xi|^2).
\end{equation}

\bigskip

For the gradient $\eta$-Ricci solitons we have:
\begin{proposition}\label{t2}
If $(g,\xi:=grad(f),\lambda,\mu)$ is a gradient $\eta$-Ricci soliton on the $n$-di\-men\-sio\-nal manifold $M$ and $\eta=df$ is the $g$-dual of $\xi$, then $\eta$ is a solution of the Schr\"{o}dinger-Ricci equation if and only if
\begin{equation}\label{e333}
d(scal)=2\mu[(scal+\lambda n+\mu|\xi|^2)df-\frac{1}{2}d(|\xi|^2)].
\end{equation}

Moreover, in this case, $scal$ is constant if and only if $\mu=0$ (which yields a gradient Ricci soliton) or $(scal+\lambda n+\mu |\xi|^2)df=\frac{1}{2}d(|\xi|^2)$.
\end{proposition}

\begin{remark}
Under the hypotheses of Proposition \ref{t2}, if the potential vector field is of constant length $k$, then (\ref{e333}) becomes:
\begin{equation}
d(scal)=2\mu(scal+\lambda n+\mu k^2)df,
\end{equation}
so the scalar curvature is constant if either the soliton is a gradient Ricci soliton or $scal=-\lambda n-\mu k^2$.
\end{remark}

\begin{remark}

i) Taking into account that for a gradient vector field $\xi$ \cite{blag}:
\begin{equation}\label{eee}
div(L_{\xi}g)=2 d(div(\xi))+2i_{Q\xi}g,
\end{equation}
the condition for the $g$-dual $\eta=df$ of the potential vector field $\xi:=grad(f)$ of a gradient $\eta$-Ricci soliton $(g,\xi,\lambda,\mu)$ to be a solution of the Schr\"{o}dinger-Ricci equation is:
\begin{equation}
d(scal+\mu |\xi|^2)=i_{Q\xi}g.
\end{equation}

In this case, $scal+\mu |\xi|^2$ is constant if and only if $\xi\in\ker Q$ and from the $\eta$-Ricci soliton equation we obtain $\nabla_{\xi}\xi=-(\lambda+\mu|\xi|^2)\xi$. Applying $\eta$ we get $\lambda+\mu|\xi|^2=-\frac{1}{2|\xi|^2}\xi(|\xi|^2)$, therefore, if the length of $\xi$ is constant (also, the scalar curvature will be constant), then $|\xi|^2=-\frac{\lambda}{\mu}$, hence $\xi$ is a geodesic vector field.

ii) If $\xi$ is an eigenvector of $Q$ (i.e. $Q\xi=a\xi$, with $a$ a smooth function), then $\eta$ is a solution of the Schr\"{o}dinger-Ricci equation if and only if $scal+\mu |\xi|^2-af$ is constant. In particular, if $\xi\in \ker Q$, then $\eta$ is a solution of the Schr\"{o}dinger-Ricci equation if and only if $\eta$ is a harmonic form.

iii) If $\eta$ is a Schr\"{o}dinger-Ricci harmonic form, then
$d(scal+\mu|\xi|^2)=2i_{Q\xi}g$.
In this case, $scal+\mu|\xi|^2$ is constant if and only if $\xi\in \ker Q$ and using the same arguments as in i) we deduce that $\xi$ is a geodesic vector field.
\end{remark}

\bigskip

Also, an exact $1$-form $df$ is harmonic if and only if the function $f$ is harmonic. In the case of a gradient $\eta$-Ricci soliton, for $\eta$ harmonic form, denoting by $\Delta_{f}:=\Delta-\nabla_{grad(f)}$ the $f$-Laplace-Hodge operator, the result stated in Theorem 3.2 from \cite{blag} becomes:
\begin{theorem}\label{t11}
Let $(g,\xi:=grad(f),\lambda,\mu)$ be a gradient $\eta$-Ricci soliton on the $n$-dimensional manifold $M$ with
$\eta=df$ the $g$-dual of $\xi$. If $\eta$ is a harmonic form, then:
\begin{equation}\label{ea}
\frac{1}{2}\Delta_{f}(|\xi|^2)=|Hess(f)|^2+\lambda |\xi|^2+\mu |\xi|^4.
\end{equation}
\end{theorem}

Using Corollary \ref{t4} we get:
\begin{corollary}
Under the hypotheses of Theorem \ref{t11}, if $M$ is of constant scalar curvature, then at least one of $\lambda$ and $\mu$ is non positive.
\end{corollary}

As a consequence for the case of gradient Ricci soliton, we have:
\begin{proposition}
Let $(g,\xi:=grad(f),\lambda)$ be a gradient Ricci soliton on the $n$-dimensional manifold $M$ of constant scalar curvature, with
$\eta=df$ the $g$-dual of $\xi$. If $\eta$ is a harmonic form, then the soliton is shrinking.
\end{proposition}
\begin{proof}
From Theorem \ref{t4} and Theorem \ref{t11} we obtain $|Hess(f)|^2+\lambda |\xi|^2=0$, hence $\lambda <0$.
\end{proof}

\begin{remark}
i) Assume that $\mu\neq 0$. If $\lambda \geq -\mu |\xi|^2$, then $\Delta_{f}(|\xi|^2)\geq 0$ and from the maximum principle follows that $|\xi|^2$ is constant in a neighborhood of any local maximum. If $|\xi|$ achieve its maximum, then $M$ is quasi-Einstein. Indeed, since $Hess(f)=0$, from the soliton equation we have $S=-\lambda g-\mu df\otimes df$. Moreover, in this case, $|\xi|^2(\lambda +\mu |\xi|^2)=0$, which implies either $\xi=0$ or $|\xi|^2=-\frac{\lambda}{\mu}\geq 0$. Since $scal+\lambda n+\mu |\xi|^2=0$ we get $scal=\lambda(1-n)$.

ii) For $\mu=0$, we get the Ricci soliton case \cite{pet2}.
\end{remark}

Computing the gradient soliton equation in $(\gamma^{\sharp},X)$, $X\in\mathfrak{X}(M)$, we obtain:
$$g(\nabla_{\gamma^{\sharp}}\xi,X)+g(Q\gamma^{\sharp},X)+\lambda g(\gamma^{\sharp},X)+\mu \eta(\gamma^{\sharp})\eta(X)=0$$
and taking $X:=\xi$ we get:
$$\frac{1}{2}\gamma^{\sharp}(|\xi|^2)+\gamma(Q\xi)+(\lambda+\mu|\xi|^2)\eta(\gamma^{\sharp})=0.$$

Therefore:
\begin{proposition}
Let $(g,\xi,\lambda,\mu)$ be an $\eta$-Ricci soliton on the $n$-dimensional manifold $M$ with $\eta$ the $g$-dual of $\xi$ and $\mu\neq 0$. If $\gamma$ is a $1$-form, then $\gamma$ is orthogonal to $\eta$ if and only if
\begin{equation}
\nabla_{\gamma^{\sharp}}\xi+Q\gamma^{\sharp}+\lambda \gamma^{\sharp}=0,
\end{equation}
hence:
\begin{equation}\label{k}
\frac{1}{2}\gamma^{\sharp}(|\xi|^2)=-\gamma(Q\xi).
\end{equation}
\end{proposition}

Some results concerning the harmonic $1$-forms on gradient $\eta$-Ricci solitons are further presented.

For two $(0,2)$-tensor fields $T_1$ and $T_2$, denote by $\langle T_1,T_2\rangle:=\sum_{1\leq i,j\leq n}T_1(E_i,E_j)T_2(E_i,E_j)$, for $\{E_i\}_{1\leq i\leq n}$ a local orthonormal frame field.

\begin{theorem}\label{t3}
Let $M$ be a compact and oriented $n$-dimensional manifold $M$, $(g,\xi:=grad(f),\lambda,\mu)$ a gradient $\eta$-Ricci soliton with
$\eta=df$ the $g$-dual of $\xi$ and $\gamma$ a $1$-form.
\begin{enumerate}
\item If $\gamma$ is orthogonal to $\eta$ and $\mu\neq 0$, then $\gamma^{\sharp}\in \ker(\nabla_{\xi}\eta+\eta\circ Q)$.
  \item If $\gamma$ is harmonic, then either we have a Ricci soliton or $\nabla_{\xi}\gamma^{\sharp}\in \ker \eta$.
  \item If $\gamma$ is exact with $\gamma=du$, then:
  \begin{equation}\label{ee}
  \int_M\langle S, div(du)\rangle =-\int_M\langle Hess(f), Hess(u)\rangle -\mu ( df| \nabla_{grad(f)}grad(u)).
  \end{equation}

  Moreover, if $\gamma$ is harmonic, the relation (\ref{ee}) becomes:
\begin{equation}
  \int_M\langle Hess(f), Hess(u)\rangle =-\mu ( df| \nabla_{grad(f)}grad(u)).
  \end{equation}
\end{enumerate}
\end{theorem}
\begin{proof}
From (\ref{k}) and using (\ref{e222}) we get:
$$0=g(\nabla_{\gamma^{\sharp}}\xi,\xi)+g(Q\xi,\gamma^{\sharp})=\xi(\eta(\gamma^{\sharp}))-\eta(\nabla_{\xi}\gamma^{\sharp})+g(\xi,Q\gamma^{\sharp})=
(\nabla_{\xi}\eta)\gamma^{\sharp}+\eta(Q\gamma^{\sharp})$$
and hence 1.

Let $\{E_i\}_{1\leq i\leq n}$ be a local orthonormal frame field with $\nabla_{E_i}E_j=0$ in a point. For any symmetric $(0,2)$-tensor field $Z$ and any $1$-form $\gamma$:
$$\langle Z,L_{\gamma^{\sharp}}g\rangle=\sum_{1\leq i,j\leq n}Z(E_i,E_j)(L_{\gamma^{\sharp}}g)(E_i,E_j)=2\sum_{1\leq i,j\leq n}Z(E_i,E_j)g(\nabla_{E_i}\gamma^{\sharp},E_j)=$$$$=2\sum_{1\leq i,j\leq n}Z(E_i,E_j)E_i(\gamma(E_j))=2\langle Z,div(\gamma)\rangle.$$

Also:
$$\langle g,L_{\gamma^{\sharp}}g\rangle=\sum_{i=1}^n(L_{\gamma^{\sharp}}g)(E_i,E_i)=2\sum_{i=1}^ng(\nabla_{E_i}\gamma^{\sharp},E_i)=2div(\gamma^{\sharp})$$
and
$$\langle df\otimes df,L_{\gamma^{\sharp}}g\rangle=
\sum_{1\leq i,j\leq n}df(E_i)df(E_j)(L_{\gamma^{\sharp}}g)(E_i,E_j)=2\sum_{1\leq i,j\leq n}df(E_i)df(E_j)g(\nabla_{E_i}\gamma^{\sharp},E_j)=$$$$=2g(\nabla_{grad(f)}\gamma^{\sharp},grad(f))=2g((\nabla_{grad(f)}\gamma)^{\sharp},(df)^{\sharp}).$$

Computing $\langle S,div(\gamma)\rangle$ by replacing $S$ from the $\eta$-Ricci soliton equation, we obtain:
$$\langle S,div(\gamma)\rangle=-\frac{1}{2}\langle Hess(f),L_{\gamma^{\sharp}}g\rangle-\lambda div(\gamma^{\sharp})-\mu g((\nabla_{grad(f)}\gamma)^{\sharp},(df)^{\sharp}).$$

For 2. we use $div(\gamma)=0=div(\gamma^{\sharp})$ and for 3. we use the fact that $\gamma^{\sharp}=grad(u)$, hence $L_{\gamma^{\sharp}}g=2Hess(u)$ and apply the divergence theorem.
\end{proof}

Since
$$\eta(\nabla_{\xi}\xi)=\frac{1}{2}\xi(|\xi|^2)$$
and for $\eta$ harmonic:
$$\int_M|Hess(f)|^2=-\mu \int_M df(\nabla_{\xi}\xi),$$
we get:

\begin{corollary}
Under the hypotheses of Theorem \ref{t3}, if $\eta$ is a harmonic form, then either we have a Ricci soliton or the potential vector field $\xi$ is of constant length. In the second case, $\eta$ is a solution of the Schr\"{o}dinger-Ricci equation and $M$ is a quasi-Einstein manifold.
\end{corollary}

We know that the Bochner formula, for an arbitrary vector field $\xi$ \cite{blag}, states:
$$\frac{1}{2}\Delta(|\xi|^2)=|\nabla \xi|^2+S(\xi,\xi)+\xi(div(\xi))$$
and taking into account that the $g$-dual $1$-form $\eta$ of $\xi$ satisfies
$$|\xi|=|\eta|, \ \ |\nabla \xi|=|\nabla \eta|, \ \ S(\xi,\xi)=S^{\sharp}(\eta,\eta), \ \ \xi(div(\xi))=\langle \Delta(\eta),\eta\rangle,$$
we have the corresponding relation for $\eta$:
\begin{equation}\label{b}
\frac{1}{2}\Delta(|\eta|^2)=|\nabla \eta|^2+S^{\sharp}(\eta,\eta)+\langle \Delta(\eta),\eta\rangle.
\end{equation}

Let $\gamma$ be a $1$-form and writing the previous relation for $\eta+\gamma$ we obtain:
$$\frac{1}{2}\Delta(\langle \eta,\gamma\rangle)=\langle \nabla \eta, \nabla \gamma\rangle+S^{\sharp}(\eta,\gamma)+\frac{1}{2}(\langle \Delta (\eta),\gamma\rangle+\langle \Delta (\gamma),\eta\rangle).$$

\begin{theorem}
Let $M$ be an $n$-dimensional manifold, $(g,\xi:=grad(f),\lambda,\mu)$ a gradient $\eta$-Ricci soliton with
$\eta=df$ the $g$-dual of $\xi$ and $\gamma$ a $1$-form. Then:
\begin{equation}\label{gh}
\frac{1}{2}\Delta(\langle df,\gamma\rangle)=\langle Hess(f), \nabla \gamma\rangle-\mu \Delta(f)\langle df,\gamma \rangle+\frac{1}{2}\langle df,\Delta (\gamma) \rangle.
\end{equation}
\begin{proof}
From (\ref{eq2}), (\ref{rr}), (\ref{eee}) and $2 div(S)=d(scal)$, we get:
$$S^{\sharp}(\eta,\gamma)=S(\xi,\gamma^{\sharp})=-\frac{1}{2}d(\Delta(f))(\gamma^{\sharp})-\mu \Delta(f) df(\gamma^{\sharp})=
-\frac{1}{2}\langle \Delta(df), \gamma\rangle-\mu \Delta(f) \langle df,\gamma\rangle,$$
hence (\ref{gh}).
\end{proof}
\end{theorem}

\begin{proposition}
Let $M$ be an $n$-dimensional manifold, $(g,\xi:=grad(f),\lambda,\mu)$ a gradient $\eta$-Ricci soliton with
$\eta=df$ the $g$-dual of $\xi$ and $\gamma$ a $1$-form.
\begin{enumerate}
\item If $\gamma$ is orthogonal to $\eta$, then $\langle Hess(f), \nabla \gamma\rangle=-\frac{1}{2}\langle df,\Delta (\gamma) \rangle$.
  \item If $\gamma$ is harmonic, then $\frac{1}{2}\Delta(\langle df,\gamma\rangle)=\langle Hess(f), \nabla \gamma\rangle-\mu \Delta(f)\langle df,\gamma \rangle$. In this case, $\langle df,\gamma\rangle$ is harmonic if and only if $\mu \Delta(f)\langle df,\gamma \rangle=\langle Hess(f), \nabla \gamma\rangle$.

  Moreover, if $\gamma$ is orthogonal to $\eta$, then $\nabla \gamma$ is orthogonal to $\nabla \eta$.
\end{enumerate}
\end{proposition}

\subsubsection*{\textbf{$L^2_f$ harmonic $1$-forms.}}

Endow the Riemannian manifold $(M,g)$ with the weighted volume form $e^{-f}dV$ and define \textit{$L^2_f$ forms} those forms $\gamma$ satisfying $\int_M |\gamma|^2e^{-f}dV < \infty$.

The most natural operator of Laplacian-type associated to the weighted manifold $(M,g,e^{-f}dV)$ is the $f$-Laplace-Hodge operator
$$\Delta_f:=\Delta-\nabla_{grad(f)}$$
which is self-adjoint with respect to this measure.

We say that a $1$-form $\gamma$ is \textit{$f$-harmonic} if
$$\Delta_f(\gamma)=0.$$

Remark that $\gamma$ is $f$-harmonic if and only if
$$\Delta(\gamma)=i_{\nabla_{\gamma}^{\sharp}\xi}g.$$

From (\ref{eq2}) and (\ref{hh}) we deduce:
\begin{proposition}
Let $(g,\xi:=grad(f),\lambda,\mu)$ be a gradient $\eta$-Ricci soliton on the $n$-dimensional manifold $M$ with
$\eta=df$ the $g$-dual of $\xi$. Then $\eta$ is an $f$-harmonic form if and only if $scal+(\mu+\frac{1}{2})|\xi|^2$ is constant.
\end{proposition}

In terms of $\Delta_f$, the relation (\ref{b}) can be written \cite{lot}:
\begin{equation}\label{pr}
\frac{1}{2}\Delta_f(|\gamma|^2)=|\nabla \gamma|^2+S_f^{\sharp}(\gamma,\gamma)+\langle \Delta_f(\gamma),\gamma\rangle,
\end{equation}
where $S_f:=Hess(f)+S$ is the Bakry-\'{E}mery Ricci tensor.

\bigskip

Using a Reilly-type formula involving the $f$-Laplacian, an interesting result was obtained in \cite{deng}, namely, if the manifold $M$ is the boundary of a compact and connected Riemannian manifold and has non negative $m$-dimensional Bakry-\'{E}mery Ricci curvature and non negative $f$-mean curvature, then either $M$ is connected or it has only two connected components, in the later case, being totally geodesic.

Another interesting topological property will be stated in the next theorem:

\begin{theorem}\label{m}
Let $(M^n, g, e^{-f}dV)$ be a complete, non compact smooth metric measure space and $(g,\xi:=grad(f),\lambda,\mu)$ a gradient $\eta$-Ricci soliton with
$\eta=df$ the $g$-dual of $\xi$. If there exists a non trivial $L^2_f$ harmonic $1$-form $\gamma_0$ such that $\lambda |\gamma_0|^2+\mu (\gamma_0(\xi))^2\leq 0$, then $M$ has finite volume and its universal covering splits isometrically into $\mathbb{R}\times N^{n-1}$.
\end{theorem}
\begin{proof}
The condition $\lambda |\gamma_0|^2+\mu (\gamma_0(\xi))^2\leq 0$ is equivalent to $S_f^{\sharp}(\gamma_0,\gamma_0)\geq 0$. From (\ref{pr}) and Lemma 3.2 from \cite{vi}:
$$|\gamma_0|\Delta_f(|\gamma_0|)\geq 0.$$

Following the same steps as in \cite{vi}, we obtain the conclusion.
\end{proof}

\begin{remark}
i) Under the hypothesis of Theorem \ref{m}, in particular, we deduce that $\gamma_0$ is $\nabla$-parallel and of constant length. Also, $\lambda\leq 0$ since in \cite{qi} was shown that $\lambda >0$ implies $M$ compact.

ii) In the Ricci soliton case, the hypothesis of Theorem \ref{m} requires that the space of $L^2_f$ harmonic $1$-forms to be nonempty and the Ricci soliton to be shrinking in order to get the same conclusion.
\end{remark}

\small{

\bigskip

\bigskip

\bigskip

\bigskip

\textit{Adara M. Blaga}

\textit{Department of Mathematics}

\textit{West University of Timi\c{s}oara}

\textit{Bld. V. P\^{a}rvan nr. 4, 300223, Timi\c{s}oara, Rom\^{a}nia}

\textit{adarablaga@yahoo.com}
}

\end{document}